\documentclass[11pt,hyp,]{nyjm}
\usepackage{hyperref}
\hypersetup{nesting=true,debug=true,naturalnames=true}
\usepackage{graphicx,amssymb,upref}

\let\<\langle
\let\>\rangle
\newcommand{\doi}[1]{doi:\,\href{http://dx.doi.org/#1}{#1}}

\let\uml\"

\title{Joint projective spectrum of $D_{\infty}^n$}



\newtheorem{thm}{Theorem}[section]
\newtheorem{defn}{Definition}[section]
\newtheorem{prop}{Proposition}[section]
\newtheorem{lem}{Lemma}[section]

\begin{document}
\author{Chen Li}

\address{School of Mathematical Sciences\\
Fudan University\\
Shanghai\\
China}
\email{22110180024@m.fudan.edu.cn}

\author{Kai Wang}

\address{School of Mathematical Sciences\\
	Fudan University\\
	Shanghai\\
	China}
\email{kwang@fudan.edu.cn}

\thanks{This work was supported by grants  NSFC (12231005, 12026250, 11722102) and the Shanghai Technology Innovation Project (21JC1400800).}

\begin{abstract}
In this paper, we compute the joint spectrum of $D_{\infty}^n$ with respect to the left regular representation and identify $n$ generators of the first de Rham cohomology group of joint resolvent set, which is induced by several central linear functionals. Through action of $D_{\infty}^n$ on 2$n$-ary trees, we obtain a self-similar realization of the group $C^*$-algebra of $D_{\infty}^n$. 
\end{abstract}

%%% ----------------------------------------------------------------------
\maketitle
%%% ----------------------------------------------------------------------
\tableofcontents
\section{Introduction}

In classical Banach algebra theory, the Gelfand theory gives a comprehensive description of spectra of operators. However, in the case of several Banach algebra elements, the (joint) spectral theory is more complicated. It is noteworthy that not only the commutativity of the tuple will influence the study, there is also a distinction between algebraic and spatial joint spectra. 

If the tuple $A=(A_1,A_2,\cdots,A_n)$ is a commutative tuple, i.e. $A_iA_j=A_jA_i$, $1\leq i,j\leq n$, J. L. Taylor defines the Taylor spectrum of the tuple by Koszul complex \cite{hormander1973introduction,taylor1970joint}. We refer the reader to \cite{curto1988applications,dosi2009frechet} for its applications in operator theory and sheaf theory.

The matter becomes difficult when the tuple is non-commuting. In \cite{yang2009projective}, R. Yang considered the invertibility of the linear pencil 
\[A(z)=z_1A_1+z_2A_2+\cdots+z_nA_n.\]
In fact, there has been increasing interest of the invertibility of $A(z)$ in fields of algebraic geometry, group theory, mathematical physics, PDEs and operator theory. We refer the reader to \cite{andersson2004functional,atkinson1972multiparameter,sleeman1978multiparameter,YANG200768} for more information. This gives rise to the following notion of projective joint spectrum.

\begin{defn}
    For a tuple $A=(A_1,A_2,\cdots,A_n)$ of elements in a unital Banach algebra $\mathcal{B}$, its projective joint spectrum $P(A)$ consists of $z\in \mathbb{C}^n$ such that $A(z)=z_1A_1+z_2A_2+\cdots+z_nA_n$ is not invertible in $\mathcal{B}$.
\end{defn}

In contrast to other notions of joint spectrum, for example, Taylor spectrum, the projective joint spectrum is novel in the sense that it is “base free". Instead of considering the invertibility of \[(A_1-z_1 I,A_2-z_2 I,\cdots, A_n-z_n I),\] it centers on the homogeneous multiparameter pencil $A(z)$. This simplifies the study in many cases. Moreover, by the homogeneity of $A(z)$, we can consider the projective joint spectrum  $p(A)$ in the complex projective space $\mathbb{P}^{n-1}=\mathbb{C}^n / \sim$ defined by $p(A)=P(A) / \sim$. In \cite{yang2009projective}, it is proved by Hartogs extension theorem that $p(A)$ is a non-trivial compact subset in $\mathbb{P}^{n-1}$.

If the Banach algebra $\mathcal{B}$ is finite dimensional, such as a matrix algebra, then the projective joint spectrum is the hypersurface $\{\det A(z)=0\}$. If the tuple is commutative, the projective joint spectrum is a union of hyperplanes, which is closely related to the case of Taylor spectrum.

The projective resolvent set $P^c(A)=\mathbb{C}^n \setminus  P(A)$ and the spectrum itself have many properties similar to those in the single-operator case. For example, it is proved that every path-connected component of $P^c(A)$ is a domain of holomorphy \cite{he2015projective}. We also refer the reader to \cite{cade2013projective,douglas2018hermitian,he2015projective,liang2018quasinilpotent,yang2009projective} on its connections with Hermitian metrics, hyperinvariant subspace problem and cyclic cohomology.

Now consider a finitely generated group $G$\footnote{In this paper, all groups are discrete locally compact groups unless otherwise stated.} with the generating set $S=\{g_1,g_2,\cdots,g_n\}$. Let $\rho$ be a unitary representation of $G$ on a Hilbert space $H$, which will be denoted by $(\rho, H)$. Let $C^*_{\rho}(G)$ denote the $C^*$-algebra generated by $A_i=\rho(g_i)$, for $i=1,2,\dots,n$. The projective joint spectrum of $G$ with respect to $\rho$, denoted by $P(A_{\rho})$, is the projective joint spectrum of the tuple $A_\rho=(I,A_1,A_2,\cdots,A_n)$.

Given two representations $(\rho_1, H_1)$ and $(\rho_2, H_2)$ of $G$, they are said to be (unitarily) equivalent if there exists a unitary map $U:H_1\rightarrow H_2$ such that
\[\rho_2(g)=U\rho_1(g)U^{-1},\qquad \forall g\in G.\]

Apparently, the projective joint spectrum is invariant for equivalent representations. Moreover, it is invariant under the weak equivalence of representations.  Let $(\pi, H)$ and $(\rho, K)$ be unitary representations of group $G$. We say that $\pi$ is weakly contained in $\rho$ if for every $\xi\in H$, every compact subset $Q$ of $G$ and every $\varepsilon>0$, there exists $\eta_1,\cdots,\eta_n\in K$ such that,
\[|\langle\pi(x)\xi,\xi\rangle-\sum\limits_{i=1}^n\langle\rho(x)\eta_i,\eta_i\rangle|<\varepsilon,\quad \forall x\in Q.\]
We shall write $\pi\prec\rho$ for this relation. If $\pi \prec\rho$ and $\rho\prec\pi$, we say that $\pi$ and $\rho$ are weakly equivalent and denote this relation by $\pi\sim\rho$. Generally speaking, it is difficult to determine whether two representations are weakly equivalent from definition. However, in \cite{dixmier1982c}, it is proved that $\pi\prec\rho$ if and only if the canonical homomorphism $\rho(g)\mapsto\pi(g), g\in G$, extends to a unital $*$-homomorphism from $C^*_{\rho}(G)$ onto $C^*_{\pi}(G)$. It implies that if $\rho(g)$ is invertible in $C^*_{\rho}(G)$, then $\pi(g)$ is invertible in $C^*_{\pi}(G)$. Moreover, $\pi\prec\rho$ implies $P(A_{\pi})\subset P(A_{\rho})$. So  $P(A_{\pi})= P(A_{\rho})$ if these two representations are weakly equivalent. 

Conversely, it is a natural question whether the projective joint spectrum determines the representation up to weak equivalence. It is often not the case. We will give some examples later in this paper. The infinite dihedral group $D_{\infty}$ is defined as the group generated by rotations and reflections of the plane that preserves the origin. Grigorchuk and Yang give a detailed description on the joint projective spectrum of $D_{\infty}$ \cite{grigorchuk2017joint}.

Throughout this article, $D_{\infty}^n$ will denote the group $\mathbb{Z}_{n}\times D_{\infty}$, which is isomorphic to $\mathbb{Z}_{n}\times \left(\mathbb{Z}_{2} * \mathbb{Z}_{2}\right)$ and has the presentation 
\begin{equation}\label{focusedgroup}
D_{\infty}^n=\left\langle a,t,\tau| a^2=t^2=\tau^n=1, a\tau=\tau a, t\tau=\tau t\right\rangle\text{.}
\end{equation}
In particular, $D_{\infty}^2$ can be realized as the group of rigid motions in 3-space consisting of rotations and reflections of the plane that preserves the origin, together with a reflection $\tau$ through the origin. 
For example,  $a(x,y,z)=(x,-y,z)$, $\tau(x,y,z)=-(x,y,z)$, and $t$ is chosen to be an involution that makes $at$ an irrational rotation in the plane $Oxy$.
We refer the readers to \cite{doi:10.1021/ed021p415.3} for more information about this group and its application on the electronic wave functions of molecules.

First, we will compute the projective joint spectrum of $D_{\infty}^n$ with respect to its left regular representation.

\begin{thm}
      If we define $P(R_{\lambda})$ as the projective joint spectrum of \[R_{\lambda}(z)=z_0 \lambda(e)+z_1\lambda(a)+z_2\lambda(t)+z_3\lambda(\tau),\] then \[P(R_{\lambda})=\bigcup\limits_{k=1}^n\bigcup\limits_{-1\leq x\leq 1}\{z\in \mathbb{C}^4:(z_0+ \omega_k z_3)^2-z_1^2-z_2^2-2z_1z_2x= 0\},\]
 where $\lambda$ is the left regular representation of $D_{\infty}^n$ and $\{\omega_0,\cdots,\omega_{n-1}\}$ is the set of  n-th roots of unity.
\end{thm}

A linear functional $\phi$ on a unital Banach algebra $\mathcal{B}$ is said to be central if $\phi(xy)=\phi(yx)$ for all $x,y\in \mathcal{B}$. In Section 3, we will concentrate on the 1-forms generated by central functionals and 
 the Maurer-Cartan form.  
Let $Tr$ and $tr$ be the canonical traces on $C^*(D_{\infty}^n)$ and $C^*(D_{\infty})$, respectively, and  $\phi_{\alpha}$ be the central linear functional on $C^*(\mathbb{Z}_n)$ defined by
\begin{equation*}
\phi_{\alpha}(\lambda_{\mathbb{Z}_n}(\tau^{\beta}))=\begin{cases}
1 & \text{if}\quad \alpha=\beta,\\
0 & \text{otherwise}.\end{cases}
\end{equation*} Then we will get the following theorem.

\begin{thm}
	The set $\{\widetilde{\phi_{\alpha}}\otimes tr\}_{\alpha=1}^{n-1}$ induces n-1 different elements besides $Tr(\omega_R(z))$ in the cohomology group $H^1_{de}(P^c(R_{\lambda}),\mathbb{C})$, where
 $\widetilde{\phi_{\alpha}}\triangleq\phi_{\alpha}-\phi_{0}$ for $1\leq\alpha\leq n-1$, and $\widetilde{\phi_0}\triangleq\phi_0$.
\end{thm}

\section{Projective joint spectrum of \texorpdfstring{$D_{\infty}^n$}{}}

In this section, we will compute the projective joint spectrum of $D_{\infty}^n$ with respect to the left regular representation.

For a discrete group $G$, the group algebra $\mathbb{C}[G]$ is the complex algebra generated by elements in $G$, i.e.
\[\mathbb{C}[G]=\{f|f=\sum\limits_{g\in G}a_g g, a_g\in \mathbb{C},
 \text{only finitely many $a_g$  non zero}\}\text{.}\] It is a $*$ -algebra under the conjugate operation defined by
\begin{equation*}
f^*=(\sum\limits_{g\in G}a_g g	)^*=\sum\limits_{g\in G}\overline{a_g}g^{-1}\text{,}
\end{equation*}
where “-" represents complex conjugation. Let $e$ denote the unit of $G$. Consider a positive definite function $tr$ on $\mathbb{C}[G]$ defined by $tr(f)=a_e$. Through the GNS construction for groups \cite{bekka2008kazhdan}, a GNS triple $(\pi_{tr},H_{tr},e_{tr})$  can be obtained, where $e_{tr}=e$
%the Hilbert space $H_{tr}=\overline{span}\{\pi_{tr}(g)e,g\in G\} $,   
and $\pi_{tr}$ is defined by $\pi_{tr}(g_1)g_2=g_1g_2$.  Let $U: H_{tr}\rightarrow l^2(G)$ be the unitary map defined by
\[U(g)=\delta_g, \quad \forall g\in G,\]
where $\delta_g$ is the function that takes value 1 at $g$ and 0 otherwise. The left regular representation  of group $G$ on $l^2(G)$, denoted by $\lambda_G$, is defined by \[\lambda_G(g)f(t)=f(g^{-1}t), \quad \forall s \in G, g \in l^2(G).\] It is well known that $\pi_{tr}$ is unitarily equivalent to the left regular representation $\lambda_G$ of group $G$ via the map $U$.

According to the presentation (\ref{focusedgroup}), $D_{\infty}^n$ consists of elements that have the form of $\tau^k(at)^j$, $\tau^kt(at)^j$ for $0\leq k\leq n-1$ and $j\in \mathbb{Z}$. From the GNS construction mentioned above, the Hilbert space $H_{tr}$ can be decomposed as $H_{tr}=\bigoplus\limits_{k=0}^{n-1}\tau^k(H\oplus tH)$, where
\begin{equation*}
H=\{f=\sum\limits_{j=-\infty}\limits^{\infty}\alpha_j(at)^k:\sum\limits_{j=-\infty}\limits^{\infty}\left| \alpha_j\right|^2< \infty\}\text{.}
\end{equation*}
Multiplication by $at$ on the Hilbert space $H$ will be denoted by $T$ in the sequel.
Through the unitary map $V\left((at)^k\right)=e^{ik\theta}$, $H$ is isomorphic to $L^2(\mathbb{T},\frac{1}{2\pi}d\theta)$ and $T$ is unitarily equivalent to the bilateral shift operator defined by
multiplication by $e^{i\theta}$.

Since projective joint spectrum is invariant with respect to unitary equivalence, we will omit writing the unitary operator $V$ in the sequel.

\begin{thm}\label{joint-spec}
	If we define $P(R_{\lambda})$ as the projective joint spectrum of \[R_{\lambda}(z)=z_0 \lambda(e)+z_1\lambda(a)+z_2\lambda(t)+z_3\lambda(\tau),\] then \[P(R_{\lambda})=\bigcup\limits_{k=1}^n\bigcup\limits_{-1\leq x\leq 1}\{z\in \mathbb{C}^4:(z_0+ \omega_k z_3)^2-z_1^2-z_2^2-2z_1z_2x= 0\},\]
 where $\lambda$ is the left regular representation of $D_{\infty}^n$ and $\{\omega_0,\cdots,\omega_{n-1}\}$ is the set of  n-th roots of unity.
\end{thm}

\begin{proof} 
	Using the orthogonal direct sum $H_{tr}=\bigoplus\limits_{k=0}^{n-1}\tau^k(H\oplus tH)$ and the unitary map 
	$W: \bigoplus\limits_{k=0}^{n-1}\tau^k(H\oplus tH)\rightarrow \bigoplus\limits_{k=0}^{n-1}(H\oplus H)$	 defined by
	\[W=diag[1,t,\eta,t\eta,\eta^2,t\eta^2,\cdots,\eta^{n-1},t\eta^{n-1}],\]
	we can easily compute that
	\begin{equation}\label{matrix-representation}
	R_{\lambda}(z)\overset{W}{\backsimeq}
	\begin{pmatrix}
	A(z) &0 &0 &\cdots &0 &z_3I_2\\
    z_3I_2 &A(z) &0 &\cdots &0 &0\\
    0 &z_3I_2 &A(z) &\cdots &0 & 0\\
    \vdots &\vdots &\vdots &\ddots &\vdots &\vdots\\
    0 &0 &0 &\cdots &A(z) &0\\
    0 &0 &0 &\cdots &z_3I_2 &A(z)\\
	\end{pmatrix}.
	\end{equation}
	Here, $A\overset{W}{\backsimeq} B$ means $A$ and $B$ is unitarily equivalent under $W$ and $A(z)=\begin{pmatrix}
	    z_0  &z_1T+z_2  \\
	z_1T^*+z_2  &z_0
	\end{pmatrix}.$
	We divide the argument into two cases.
	
	\noindent Case \uppercase\expandafter{\romannumeral1}
	
	If $z_3=0$, it is trivial since $R_{\lambda}(z)$ is invertible if and only if $A(z)$ is invertible. In this case, the projective joint spectrum $P(R_{\lambda})$ equals \[P(R_{\lambda})=\bigcup\limits_{-1\leq x\leq 1}\{z\in \mathbb{C}^4:z_0^2-z_1^2-z_2^2-2z_1z_2x= 0, z_3=0\}\] by \cite[Theorem 1.1]{grigorchuk2017joint}.

	\noindent Case \uppercase\expandafter{\romannumeral2}

	If $z_3\ne 0$, by multiplying $\begin{pmatrix}
	0  &I_{2n-2}\\
	I_2  &0\\
	\end{pmatrix}$ on the right, we turn $R_{\lambda}(z)$ into \begin{equation}\label{Rtilde}
	    \widetilde{R(z)}=\begin{pmatrix}
	z_3I_2 &A(z) &0 &\cdots &0 &0\\
    0 &z_3I_2 &A(z) &\cdots &0 &0\\
    0 &0 &z_3I_2 &\cdots &0 & 0\\
    \vdots &\vdots &\vdots &\ddots &\vdots &\vdots\\
    0 &0 &0 &\cdots &z_3I_2 &A(z)\\
    A(z) &0 &0 &\cdots &0 &z_3I_2\end{pmatrix}.
	\end{equation} By the Schur complement trick in \cite{haynsworth1968determination}, $R_{\lambda}(z)$ is invertible if and only if $z_3^n-(-1)^nA(z)^n$ is invertible.
	
	Using the spectral theorem for normal operators in \cite[Theorem 9.2.2]{conway2019course}, we can write  \begin{equation}\label{spec-reso}
	T=\int_{\mathbb{T}}\lambda dE(\lambda),
	\end{equation} where $dE(\lambda)$ is the spectral measure of $T$. 
	Noting that
	\[z_3^n-(-1)^nA(z)^n=(-1)^{n+1}\prod\limits_{k=1}^n(z_3\omega_k+A(z))\]
and $z_3\omega_k+A(z)$ mutually commutes for $1\leq k\leq n$, it follows that $z_3^n-(-1)^nA(z)$  is invertible if and only
	if $z_3\omega_k+A(z)$ is invertible for each $k$, or more precisely, by (\ref{spec-reso})
	\[(z_0+ \omega_k z_3)^2-z_1^2-z_2^2-z_1z_2(\lambda+\Bar{\lambda})\ne 0,\quad 
	\forall  \lambda \in \mathbb{T}, 1\leq k\leq n.\] Letting $\lambda=e^{i\theta}$ for $ \theta \in [ 0,2\pi)$, $x=\cos\theta$, we have \[P^c(R_{\lambda})=\bigcap\limits_{k=1}^n\bigcap\limits_{-1\leq x\leq 1}\{z\in \mathbb{C}^4:(z_0+\omega_k z_3)^2-z_1^2-z_2^2-2z_1z_2x\ne 0\},\] and the theorem is proved by taking its complement.
	
\end{proof}

In the sequel, we set
\begin{equation}\label{suoxie}
\begin{aligned}
G_{\theta}^{k}(z)&=(z_0+\omega_k z_3)^2-z_1^2-z_2^2-2z_1z_2\cos\theta.\\
\end{aligned}
\end{equation}

\noindent \textbf{Remark}\quad It is proved in \cite{grigorchuk2017joint} that the Koopman representation and the left regular representation of $D_{\infty}$ are weakly equivalent. Without much effort, we can obtain by Theorem \cite[Proposition F.3.2]{bekka2008kazhdan} that the result also holds for $D_{\infty}^n$.

\section{Trace of Maurer-Cartan form and de Rham cohomology group}

In \cite{yang2009projective},
for a tuple $A=(A_1,A_2,\dots,A_n)$, the Maurer-Cartan form $\omega_A$ is an operator-valued 1-form defined by 
\begin{equation*}
\omega_A(z)=A(z)^{-1}dA(z)=\sum\limits_{i=1}^n A(z)^{-1}A_idz_i, \qquad \forall z\in P^c(A).
\end{equation*}
A linear functional $\phi$ on a unital Banach algebra $\mathcal{B}$ is said to be central if $\phi(xy)=\phi(yx)$ for any $ x,y\in \mathcal{B}$. In \cite[Theorem 3.2]{yang2009projective}, it is proved that if $\phi$ is central and $\phi(I)\ne 0$, then $\phi(\omega_A)$ is a non-trivial element in the de Rham cohomology group $H^1_{de}(P^c(A),\mathbb{C})$. This section is devoted to studying 1-forms induced by different central linear functionals.

For a discrete group $G$, it is well known that its reduced group $C^*$-algebra $C_r^*(G)$ admits a canonical tracial state  
\begin{equation}\label{Ctrformula}
tr(a)=\left<a\delta_e,\delta_e\right> ,\qquad \forall a \in C_r^*(G).
\end{equation}
In \cite[Corollary 4.3]{breuillard2017c}, it is proved that the reduced group $C^*$-algebra $C_r^*(G)$ has only one tracial state if and only if $G$ is amenable or none of its normal subgroup are amenable. Since $D_{\infty}^n$ itself is amenable, the canonical trace $Tr$ is the unique tracial state on $C_r^*(D_{\infty}^n)$.

Since $D_{\infty}=\mathbb{Z}\rtimes\mathbb{Z}_2$ and $D_{\infty}^n=\mathbb{Z}_n\times D_{\infty}$, the inclusion maps
\[\mathbb{Z}\hookrightarrow D_{\infty}\hookrightarrow D_{\infty}^n\]  induce the inclusion maps of group $C^*$-algebras. Based on this observation, we have the following proposition.

\begin{prop}
	The canonical traces on $C^*_r(\mathbb{Z}), C^*_r(D_{\infty})$ and $ C^*_r(D_{\infty}^n)$ coincide by restriction.
\end{prop}

\begin{proof}
  We first denote the canonical traces on these $C^*$-algebras by $tr_{\mathbb{Z}}$, $tr$ and $Tr$. 
	
	By the form of elements in $D_{\infty}$, $l^2(D_\infty)\cong l^2(\mathbb{Z})\oplus l^2(t\mathbb{Z})$. This implies that any $a\in C^*_r(\mathbb{Z})$ can be treated as $a\oplus 0\in C^*_r(D_{\infty})$. Thus
	\[tr_\mathbb{Z}(a)=\left<a\delta_{e_\mathbb{Z}},\delta_{e_\mathbb{Z}}\right>=\left<(a\oplus 0)\delta_{e_{D_{\infty}}},\delta_{e_{D_{\infty}}}\right>=tr(a\oplus 0),\]
	which leads to $tr|_{\mathbb{Z}}=tr_{\mathbb{Z}}$.
	
	Since \[C^*_r(D_{\infty}^n)=\overline{C^*_r(\mathbb{Z}_n)\otimes C^*_r(D_{\infty})}^{B(l^2(D_{\infty}^n))}\]
	and $\delta_{e_{D_{\infty}^n}}=\delta_{e_{\mathbb{Z}_n}}\otimes \delta_{e_{D_{\infty}}}$, for any $a\in C^*_r(D_{\infty})$,
	
	\begin{align*}
		tr(a)&=\left<a\delta_{e_{D_{\infty}}},\delta_{e_{D_{\infty}}}\right>=\left<a\delta_{e_{D_{\infty}}},\delta_{e_{D_{\infty}}}\right>\left<\delta_{e_{\mathbb{Z}_n}},\delta_{e_{\mathbb{Z}_n}}\right>\\
		&=\left<(a\otimes 1)\delta_{e_{D_{\infty}^n}},\delta_{e_{D_{\infty}^n}}\right>\\
		&=Tr(a\otimes 1)
	\end{align*}
	Therefore, $Tr|_{D_{\infty}}=tr$.
\end{proof}
Due to this proposition, we will not distinguish the canonical traces on these groups and denote them all by $tr$.

The group Von Neumann algebra $L(G)$ is the closure of $\lambda_G(\mathbb{C}[G])$ with respect to the weak operator topology in the Hilbert space. Thus, formula (\ref{Ctrformula}) can be naturally extended to $L(G)$. In the previous section, we obtained the matrix representation of $R_{\lambda}(z)$ by \eqref{matrix-representation}. In order to compute the trace of the Maurer-Cartan form, we can canonically define the extended trace $ \widetilde{tr}$ on $2n\times 2n$ matrices with $L(\mathbb{Z})$ entries by
\begin{equation}\label{trace-formula}
\widetilde{tr}((a_{ij})_{i,j=1}^{2n}):=\frac{1}{2n}tr(\sum\limits_{i=1}^{2n} a_{ii}).
\end{equation}
The sign "$\sim$" will be omitted if there is no confusion.

By the observation in \cite{grigorchuk2017joint}, we have
\begin{equation}\label{operator-resolution}
tr(dE(\lambda))=tr(dE(e^{i\theta}))=dtr(E(e^{i\theta}))=\frac{1}{2\pi}d\theta,
\end{equation} where  $E(\lambda)$ is the spectral measure in (\ref{spec-reso}). 

For $D_{\infty}^n$, the Maurer-Cartan form is
\begin{equation}\label{M-Cform}
\omega_R(z)=R^{-1}(z)dR(z)=R^{-1}(z)(dz_0+adz_1+tdz_2+\tau dz_3),
\end{equation} where $R(z)$ is defined in (\ref{matrix-representation}).

\begin{lem}\label{deno-fac}
    Suppose $z=(z_0,z_1,z_2,z_3)\in \mathbb{C}^4$, and let $f$ and $g$ denote $z_1e^{i\theta}+z_2$ and  $z_1e^{-i\theta}+z_2$, respectively. If $a_{11}^{(n)}=\sum\limits_{k=0}^{[\frac{n}{2}]}z_0^{n-2k}\binom{n}{n-2k}f^kg^k,$ and $
     a_{12}^{(n)}=\sum\limits_{k=0}^{[\frac{n-1}{2}]}z_0^{n-2k-1}\binom{n}{n-2k-1}f^kg^k,$ where $\binom{n}{m}=\frac{n!}{(n-m)!m!}$, then we have the following factorization:
     \[(z_3^n-(-1)^n a_{11}^{(n)})^2-(a_{12}^{(n)})^2fg=\prod\limits_{k=1}^nG_{\theta}^k(z).\]
\end{lem}

\begin{proof}
    Using direct computation,
    \begin{align*}
        LHS=&\left(z_3^n-(-1)^n(a_{11}^{(n)}+a_{12}^{(n)}\sqrt{fg})\right)\left(z_3^n-(-1)^n(a_{11}^{(n)}-a_{12}^{(n)}\sqrt{fg})\right)\\
        =&\left((a_{11}^{(n)}+a_{12}^{(n)}\sqrt{fg})-(-z_3)^n\right)\left((a_{11}^{(n)}-a_{12}^{(n)}\sqrt{fg})-(-z_3)^n\right)\\
        =&\left((z_0+\sqrt{fg})^n-(-z_3)^n\right)\left((z_0-\sqrt{fg})^n-(-z_3)^n\right)\\
        =&\prod_{k=0}^{n-1}(z_0+\sqrt{fg}+z_3\omega_k)(z_0-\sqrt{fg}+z_3\omega_k)\\
        =&\prod_{k=0}^{n-1}\left((z_0+z_3\omega_k)^2-z_1^2-z_2^2-2z_1z_2\cos\theta\right)\\
        =&RHS
    \end{align*}
\end{proof}
\begin{prop}\label{traceM}
	The trace of Maurer-Cartan form $\omega_R(z)$ is 
	\[\widetilde{tr}(\omega_{R}(z))=\displaystyle d\bigg(\frac{1}{4n\pi}\int_0^{2\pi}\log(\prod\limits_{k=1}^nG_{\theta}^{k}(z))d\theta\bigg), \,z\in P^c(R_{\lambda}),\]
	 where $G_{\theta}^{k}(z)$ is defined in (\ref{suoxie}), and $\displaystyle d$ stands for 
	\begin{align*}
	\frac{\partial}{\partial z_0}dz_0+\frac{\partial}{\partial z_1}dz_1+\frac{\partial}{\partial z_2}dz_2+\frac{\partial}{\partial z_3}dz_3.
	\end{align*}
\end{prop}
\begin{proof}
	
	As in the proof of Theorem \ref{joint-spec}, the argument will be divided into two parts.
 
	\noindent Case \uppercase\expandafter{\romannumeral1}
	
	If $z_3=0$ , then $\omega_R(z)=(z_0+z_1 a+z_2 t)^{-1}(dz_0+adz_1+tdz_2)$. In this case, we can directly use \cite[Proposition 3.2]{grigorchuk2017joint} to get
	\[\widetilde{tr}(\omega_R(z))=d\left(\frac{1}{4\pi}\int_0^{2\pi}\log(z_0^2-z_1^2-z_2^2-2z_1z_2\cos\theta)d\theta\right).\]
	\noindent Case \uppercase\expandafter{\romannumeral2}
	
	For $z_3\ne 0$,  we will first compute the inverse of $\widetilde{R(z)}$ defined in (\ref{Rtilde}).

  Using the Schur complement as in the proof of Theorem \ref{joint-spec}, we define a $2(n-k)\times 2(n-k)$ matrix $S_{n-k}(z)$ by
  \[S_{n-k}(z)=\begin{pmatrix}
      z_3I_2 &A(z) &0 &\cdots &0 &0\\
    0 &z_3I_2 &A(z) &\cdots &0 &0\\
    0 &0 &z_3I_2 &\cdots &0 & 0\\
    \vdots &\vdots &\vdots &\ddots &\vdots &\vdots\\
    0 &0 &0 &\cdots &z_3I_2 &A(z)\\
    \frac{(-1)^k}{z_3^k}A(z)^{k+1} &0 &0 &\cdots &0 &z_3I_2 
  \end{pmatrix}.\]
 Thus, $S_{n-1}(z)$ is the Schur complement of $\widetilde{R(z)}$, and $S_{n-k}(z)$ is that of $S_{n-k+1}(z)$.

 To perform the computation, we should first focus on $S_1(z)^{-1}$ and $A(z)^n$.
 We obtain by induction that $A(z)^n=\begin{pmatrix}
     a_{11}^{(n)} &a_{12}^{(n)}f\\
    a_{12}^{(n)}g &a_{11}^{(n)}\\
 \end{pmatrix}$, where $a_{11}^{(n)}$ and $a_{12}^{(n)}$ are defined in Lemma \ref{deno-fac}.

Using Lemma \ref{deno-fac}, for $S_1(z)=\frac{1}{z_3^{n-2}}(z_3^n-(-1)^nA(z)^n)$,
{\small \begin{align*}
   S_1(z)^{-1}&=\frac{z_3^{n-2}}{(z_3^n-(-1)^n a_{11}^{(n)})^2-(a_{12}^{(n)})^2fg}\begin{pmatrix}
     z_3^n-(-1)^na_{11}^{(n)} &(-1)^{n}a_{12}^{(n)}f\\
     (-1)^na_{12}^{(n)}g &z_3^n-(-1)^na_{11}^{(n)}
 \end{pmatrix}\\
 &=\frac{z_3^{n-2}}{\prod\limits_{k=1}^nG_{\theta}^k(z)}\begin{pmatrix}
     z_3^n-(-1)^na_{11}^{(n)} &(-1)^{n}a_{12}^{(n)}f\\
     (-1)^na_{12}^{(n)}g &z_3^n-(-1)^na_{11}^{(n)}
 \end{pmatrix}.
\end{align*}}

 By induction, we have
 $S_k(z)^{-1}=(s_{ij})_{i,j=1}^k$
 with \begin{align*}
     s_{11}&=z_3^{-1}+\frac{(-1)^n}{z_3^n}S_1(z)^{-1}A(z)^n,\\
     s_{1k}&=\frac{(-1)^{k-1}}{z_3^{k-1}}S_1(z)^{-1}A(z)^{k-1},\\
     s_{1j}&=\frac{(-1)^{j-1}}{z_3^j}A(z)^{j-1}+\frac{(-1)^{n+j-1}}{z_3^{n+j-1}}S_1(z)^{-1}A(z)^{n+j-1},\quad 2\leq j\leq k-1\\
     s_{i1}&=\frac{(-1)^{n-i+1}}{z_3^{n-i+1}}S_1(z)^{-1}A(z)^{n-i+1},\quad i>1
 \end{align*}
 and $(s_{ij})_{i,j=2}^k=S_{k-1}(z)^{-1}$.

 Thus $\widetilde{R(z)}^{-1}=(r_{ij})_{i,j=1}^n$ with
 \begin{align*}
      r_{11}&=z_3^{-1}+\frac{(-1)^n}{z_3^n}S_1(z)^{-1}A(z)^n,\\
     r_{1n}&=\frac{(-1)^{n-1}}{z_3^{n-1}}S_1(z)^{-1}A(z)^{n-1},\\
     r_{1j}&=\frac{(-1)^{j-1}}{z_3^j}A(z)^{j-1}+\frac{(-1)^{n+j-1}}{z_3^{n+j-1}}S_1(z)^{-1}A(z)^{n+j-1},\quad 2\leq j\leq n-1\\
     r_{i1}&=\frac{(-1)^{n-i+1}}{z_3^{n-i+1}}S_1(z)^{-1}A(z)^{n-i+1},\quad i>1
 \end{align*}
 and $(r_{ij})_{i,j=2}^n=S_{n-1}(z)^{-1}$.

 Therefore, by \eqref{trace-formula}, \eqref{operator-resolution} and the matrices of $A(z)^{n-1}$ and $S_1(z)^{-1}$,
 \noindent\begin{align*}
     &\widetilde{tr}(R(z)^{-1})=\frac{1}{2n}tr\left(\frac{n(-1)^{n-1}A(z)^{n-1}S_1(z)^{-1}}{z_3^{n-1}}\right)\\
     &=\frac{1}{2\pi}\int_0^{2\pi}\frac{(-1)^{n-1}((z_3^n-(-1)^na_{11}^{(n)})a_{11}^{(n-1)}+(-1)^na_{12}^{(n)}a_{12}^{(n-1)}fg)}{\prod\limits_{k=1}^nG^k_{\theta}(z)}d\theta\\
     &=\frac{1}{4n\pi}\int_0^{2\pi}\sum\limits_{k=1}^{n}\frac{2(z_0+\omega_k z_3)}{G^k_{\theta}(z)}d\theta.
 \end{align*}

Similarly, we have that
	\begin{align*}
		\widetilde{tr}(R^{-1}(z)a)&=-\frac{1}{4n\pi}\int_0^{2\pi}\sum\limits_{k=1}^{n}\frac{2(z_1+z_2cos\theta)}{G^k_{\theta}(z)}d\theta,\\
		\widetilde{tr}(R^{-1}(z)t)&=-\frac{1}{4n\pi}\int_0^{2\pi}\sum\limits_{k=1}^{n}\frac{2(z_2+z_1cos\theta)}{G^k_{\theta}(z)}d\theta,
	\end{align*}
	and
	
	\begin{equation*}
	\widetilde{tr}(R^{-1}(z)\tau)=\frac{1}{4n\pi}\int_0^{2\pi}\sum\limits_{k=1}^{n}\frac{2\omega_k(z_0+\omega_k z_3)}{G^k_{\theta}(z)}d\theta.
	\end{equation*}
	By  (\ref{M-Cform}), summing all the results above, we obtain
	\begin{align*}
	\widetilde{tr}(\omega_{R}(z))= d\bigg(\frac{1}{4n\pi}\int_0^{2\pi}\log(\prod\limits_{k=1}^nG_{\theta}^{k}(z))d\theta\bigg).
	\end{align*}
\end{proof}

At the beginning of this section, we mentioned that a central linear functional $\phi$ will induce a non-trivial closed 1-form in the de Rham cohomology group of the projective resolvent set if $\phi(I)\ne 0$. In \cite[Corollary 4.4]{grigorchuk2017joint}, it is shown by the de Rham duality theorem that the  $H^1_{de}(P^c(R_{\lambda_{D_\infty}}),\mathbb{C})$ is generated by $\frac{1}{2\pi}tr(\omega_{R_{\lambda_{D_\infty}}})$.
So it is a natural question what the case is for $D_{\infty}^n$. In fact, there are other central linear functionals that yield nontrivial elements in  the cohomology group.

It is well known that  for any discrete groups $G$ and $H$
\begin{align*}
	C^*(G\times H)&=C^*(G)\otimes_{max} C^*(H),\\
	C^*_r(G\times H)&=C^*_r(G)\otimes_{min} C^*_r(H).
\end{align*}
Moreover, if one of the groups, for example $G$, is amenable, then the group $C^*$-algebra $C^*(G)$ is nuclear and  is isomorphic to $C^*_r(G)$.  Back to our question, since both $D_{\infty}$ and $\mathbb{Z}_n$ are amenable, $C^*(\mathbb{Z}_n\times D_{\infty})$ is the closure of  algebraic tensor product of $C^*(D_{\infty})$ and $C^*(\mathbb{Z}_n)$ in $B(l^2(D_{\infty}^n))$. Therefore, a tensor product of two linear functionals, on the two respective $C^*$-algebras, extends to a linear functional on $C^*(D_{\infty}^n)$.

Apparently, $\dim C^*(\mathbb{Z}_n)=n$. For every element $x\in C^*(\mathbb{Z}_n)$, it takes the form \[x=\sum\limits_{\beta=0}^{n-1}b_{\beta} \lambda_{\mathbb{Z}_n}(\tau^{\beta}), b_{\beta}\in \mathbb{C}.\] Since $\lambda_{\mathbb{Z}_n}(\tau)$ has the matrix representation 
\begin{equation*}
	    \lambda_{\mathbb{Z}_n}(\tau)=\begin{pmatrix}
	0 &1 &0 &\cdots &0 &0\\
    0 &0 &1 &\cdots &0 &0\\
    0 &0 &0 &\cdots &0 & 0\\
    \vdots &\vdots &\vdots &\ddots &\vdots &\vdots\\
    0 &0 &0 &\cdots &0 &1\\
    1 &0 &0 &\cdots &0 &0\end{pmatrix},
	\end{equation*}
$C^*(\mathbb{Z}_n)$ can be identified as a subalgebra of $M_n(\mathbb{C})$ by the map $\varphi$ defined by $\varphi(\lambda_{\mathbb{Z}_n}(\tau^{\beta}))=\sum\limits_{i=n-\beta}^{n-1}E^{(n)}_{i(i-n+\beta)}+\sum\limits_{i=0}^{n-1-\beta}E^{(n)}_{i(i+\beta)}$. Here $\{E_{ij}^{(n)}: 0\leq i\leq n-1, 0\leq j\leq n-1\}$ is the generating set for $M_{n}(\mathbb{C})$, where $E_{ij}^{(n)}$ denotes the matrix with a 1 at the $i$-th row and $j$-th column and 0 elsewhere.

We first define $n$ linear functionals $\phi_{\alpha}$, $0\leq\alpha\leq n-1$ , on  $M_{n}(\mathbb{C})$ by\begin{equation*}
\phi_{\alpha}(E_{ij}^{(n)})=\begin{cases}
\frac{1}{n} & \text{if}\quad j-i=\alpha  ,\\
\frac{1}{n} & \text{if}\quad i-j=n-\alpha,\\
0 & \text{otherwise}.\end{cases}
\end{equation*}
It follows that every linear functional on $C^*(\mathbb{Z}_n)$ is a complex linear combination of $\{\phi_{\alpha}\}_{\alpha=0}^{n-1}$ by restriction.

In fact, $\phi_0\otimes tr$ is just the trace on $C^*(D_{\infty}^n)$. Now define $\widetilde{\phi_{\alpha}}\triangleq\phi_{\alpha}-\phi_0$ for $1\leq\alpha\leq n-1$, and $\widetilde{\phi_0}\triangleq\phi_0$. One can check that $\widetilde{\phi_{\alpha}}$ is not a trace for $1\leq\alpha\leq n-1$, as it takes the value -1 at the identity. We first present the following observation.
\begin{lem}\label{new-trace-formula}
The formula of $(\phi_\alpha\otimes tr)((a_{ij})_{i,j=0}^{2n-1})$ is
 \[(\phi_\alpha\otimes tr)((a_{ij})_{i,j=0}^{2n-1})=
\frac{1}{2n}tr(\sum\limits_{i=2n-2\alpha}^{2n-1}a_{i(i-2n+2\alpha)}+\sum\limits_{i=0}^{2n-2\alpha-1}a_{i(2\alpha+i)}),\] $\forall 1\leq\alpha\leq n-1$, $(a_{ij})_{i,j=0}^{2n-1}\in  C^*(D_{\infty}^n)$. 
\end{lem}
\begin{proof}
    Let $\{E_{ij}: 0\leq i\leq 2n-1, 0\leq j\leq 2n-1\}$ be the generating set for $M_{2n}(L(\mathbb{Z}))$, and $\{\widetilde{E}_{ij}: 0\leq i\leq 1, 0\leq j\leq 1\}$ for $M_{2}(L(\mathbb{Z}))$. Here

    The following equation can be easily verified:
    \[E_{ij}=E^{(n)}_{[\frac{i}{2}][\frac{j}{2}]}\otimes \widetilde{E}_{(2\{\frac{i}{2}\})(2\{\frac{j}{2}\})},\]
    where the brackets $[x]$ denote the floor operation and $\{x\}$ denote the fractional part operation.

    Thus, 
    \begin{align*}
        (\phi_\alpha\otimes tr)(E_{ij})&=\phi_{\alpha}(E^{(n)}_{[\frac{i}{2}][\frac{j}{2}]})tr(\widetilde{E}_{(2\{\frac{i}{2}\})(2\{\frac{j}{2}\})})\\
        &=\frac{1}{n}\left(\delta_{\alpha([\frac{j}{2}]-[\frac{i}{2}])}+\delta_{(n-\alpha)([\frac{i}{2}]-[\frac{j}{2}])}\right)\frac{1}{2}tr(\Delta_{(2\{\frac{i}{2}\})(2\{\frac{j}{2}\})})\\
        %&=\frac{1}{2n}\left(\delta_{\alpha([\frac{j}{2}]-[\frac{i}{2}])}+\delta_{(n-\alpha)([\frac{j}{2}]-[\frac{i}{2}])}\right)tr(\Delta_{(2\{\frac{i}{2}\})(2\{\frac{j}{2}\})})\\
        &=\begin{cases}
\frac{1}{2n} & \text{if}\quad j-i=2\alpha  ,\\
\frac{1}{2n} & \text{if}\quad i-j=2n-2\alpha,\\
0 & \text{otherwise},\end{cases}
    \end{align*}
    where $\delta_{xy}$ is the kronecker symbol and 
    \begin{equation*}
\Delta_{xy}=\begin{cases}
I(L^2(\mathbb{Z})) & \text{if} \quad x=y, \\
0 & \text{otherwise}.\end{cases}
\end{equation*}
Therefore,
\begin{align*}
    (\phi_\alpha\otimes tr)((a_{ij})_{i,j=0}^{2n-1})&=\sum\limits_{i,j=0}^{n-1}\phi_\alpha(E^{(n)}_{ij})tr\left(\begin{pmatrix}
        a_{(2i)(2j)} & a_{(2i)(2j+1)}\\
        a_{(2i+1)(2j)} & a_{(2i+1)(2j+1)}
    \end{pmatrix}\right)\\
    &=\sum\limits_{i=0}^{n-1-\alpha}\frac{1}{2n}tr(a_{(2i)(2i+2\alpha)}+a_{(2i+1)(2i+2\alpha+1)})+\\
    &\sum\limits_{i=n-\alpha}^{n-1}\frac{1}{2n}tr(a_{(2i)(2i+2\alpha-2n)}+a_{(2i+1)(2i+2\alpha-2n+1)})\\
    &=\frac{1}{2n}tr(\sum\limits_{i=2n-2\alpha}^{2n-1}a_{i(i-2n+2\alpha)}+\sum\limits_{i=0}^{2n-2\alpha-1}a_{i(2\alpha+i)}).
\end{align*}
\end{proof}
\begin{thm}
	$\{\widetilde{\phi_{\alpha}}\otimes tr\}_{\alpha=1}^{n-1}$ induces n-1 different 1-forms besides $Tr(\omega_R(z))$ in the cohomology group $H^1_{de}(P^c(R_{\lambda_{D_{\infty}^n}}),\mathbb{C})$, where $Tr$ is the canonical trace on $C^*(D_{\infty}^n)$.
\end{thm}

\begin{proof}
    
       Firstly, using the formula in Lemma \ref{new-trace-formula}, we get for $1\leq \alpha \leq n-1$
	\begin{align*}
	(\phi_{\alpha}&\otimes tr)(R(z)^{-1})=\frac{1}{2\pi}\int_0^{2\pi}\frac{(-1)^{n-\alpha-1}a_{11}^{(n-\alpha-1)}z^{n+\alpha}_3d\theta}{\prod\limits_{k=1}^nG_{\theta}^k(z)}\\&+\frac{1}{2\pi}\int_0^{2\pi}\frac{(-1)^{\alpha}(a_{11}^{(n)}a_{11}^{(n-\alpha-1)}-a_{12}^{(n)}a_{12}^{(n-\alpha-1)}fg)z_3^{\alpha}d\theta}{\prod\limits_{k=1}^nG_{\theta}^k(z)},
	\end{align*}
	
	\begin{align*}
	(\phi_{\alpha}&\otimes tr)(R(z)^{-1}a)=\frac{1}{2\pi}\int_0^{2\pi}\frac{(-1)^{n-\alpha-1}(z_1+z_2cos\theta)a_{12}^{(n-\alpha-1)}z^{n+\alpha}_3d\theta}{\prod\limits_{k=1}^nG_{\theta}^k(z)}\\&+\frac{1}{2\pi}\int_0^{2\pi}\frac{(-1)^{\alpha}(z_1+z_2cos\theta)(a_{11}^{(n)}a_{12}^{(n-\alpha-1)}-a_{12}^{(n)}a_{11}^{(n-\alpha-1)})z_3^{\alpha}d\theta}{\prod\limits_{k=1}^nG_{\theta}^k(z)},
	\end{align*}
\begin{align*}
	(\phi_{\alpha}&\otimes tr)(R(z)^{-1}t)=\frac{1}{2\pi}\int_0^{2\pi}\frac{(-1)^{n-\alpha-1}(z_2+z_1cos\theta)a_{12}^{(n-\alpha-1)}z^{n+\alpha}_3d\theta}{\prod\limits_{k=1}^nG_{\theta}^k(z)}\\&+\frac{1}{2\pi}\int_0^{2\pi}\frac{(-1)^{\alpha}(z_2+z_1cos\theta)(a_{11}^{(n)}a_{12}^{(n-\alpha-1)}-a_{12}^{(n)}a_{11}^{(n-\alpha-1)})z_3^{\alpha}d\theta}{\prod\limits_{k=1}^nG_{\theta}^k(z)},
	\end{align*}
 and
\begin{align*}
	(\phi_{\alpha}&\otimes tr)(R(z)^{-1}\tau)=\frac{1}{2\pi}\int_0^{2\pi}\frac{(-1)^{n-\alpha}a_{11}^{(n-\alpha)}z^{n+\alpha-1}_3d\theta}{\prod\limits_{k=1}^nG_{\theta}^k(z)}\\&+\frac{1}{2\pi}\int_0^{2\pi}\frac{(-1)^{\alpha+1}(a_{11}^{(n)}a_{11}^{(n-\alpha)}-a_{12}^{(n)}a_{12}^{(n-\alpha)}fg)z_3^{\alpha-1}d\theta}{\prod\limits_{k=1}^nG_{\theta}^k(z)}.
	\end{align*}
 
     Fix $k\in \{0, 1, \cdots, n-1\}$ and choose a closed path \[\gamma_k=\{((b+1)\omega_k,\omega_k c_{nk} e^{it}+\omega_k,0,b): 0\leq t \leq 2\pi\},\] where $c_{nk}=\frac{1}{2}\min\{\min\limits_{j\ne k} b\lvert \omega_k-\omega_j\rvert,\min\limits_{j}\lvert (b+2)\omega_k-b\omega_j\rvert\},$ and $b$ is a non-negative constant. For every $-1\leq x\leq 1$ and $j\in \{0, 1, \cdots, n-1\}$, we have that
		\begin{align*}
			&(z_0+\omega_j z_3)^2-z_1^2-z_2^2-2z_1z_2x\\&=(b\omega_k-b\omega_j-\omega_kc_{nk}e^{it})((b+2)\omega_k-b\omega_j+\omega_kc_{nk}e^{it})\ne 0,
		\end{align*}
	This means $\gamma_k \in P^c(R_{\lambda_{D_{\infty}^n}})$ by Theorem \ref{joint-spec}.
     
        On $\gamma_k$, for $0\leq\alpha\leq n-1$ the path integral
 \begin{equation}\label{need-more-cal}
        \begin{aligned}
			&\frac{1}{2\pi i}\int_{\gamma_k}(\phi_{\alpha}\otimes tr)(\omega_R(z))\\
   &=\frac{1}{2\pi i}\int_{\gamma_k}\frac{z_3^{\alpha}}{2\pi }\int_0^{2\pi}\frac{(-1)^{n-\alpha-1}(f+g)a_{12}^{n-\alpha-1}z^n_3}{\prod\limits_{j=1}^n G_{\theta}^{k}(z)}d\theta dz_1\\
 &-\frac{1}{2\pi i}\int_{\gamma_k}\frac{z_3^{\alpha}}{2\pi }\int_0^{2\pi}\frac{(-1)^{\alpha+1}(f+g)(z_0^2-z_1^2)^{n-\alpha-1}a_{12}^{\alpha+1}}{\prod\limits_{j=1}^n G_{\theta}^{k}(z)}d\theta dz_1\\
 &=\frac{1}{4\pi }\int_{0}^{2\pi}\frac{b^{n+\alpha}\omega_k^{-\alpha-1}((b+2+c_{nk}e^{it})^{n-\alpha-1}-(b-c_{nk}e^{it})^{n-\alpha-1})}{\prod\limits_{j=1}^n((b+2)\omega_k-b\omega_j+\omega_kc_{nk}e^{it})\prod\limits_{j\ne k}(b\omega_k-b\omega_j-\omega_kc_{nk}e^{it})}dt\\
 &+\frac{1}{4\pi }\int_{\gamma_k}\frac{b^{\alpha}\omega_k^{-\alpha-1}(b+2+ c_{nk}e^{it})^{\alpha+1}((b+1)^2-(c_{nk}e^{it}+1)^2)}{\prod\limits_{j=1}^n((b+2)\omega_k-b\omega_j+\omega_kc_{nk}e^{it})\prod\limits_{j\ne k}(b\omega_k-b\omega_j-\omega_kc_{nk}e^{it})}dt  \\
 &-\frac{1}{4\pi }\int_{\gamma_k}\frac{b^{\alpha}\omega_k^{-\alpha-1}(b-c_{nk}e^{it})^{\alpha+1}((b+1)^2-(c_{nk}e^{it}+1)^2)}{\prod\limits_{j=1}^n((b+2)\omega_k-b\omega_j+\omega_kc_{nk}e^{it})\prod\limits_{j\ne k}(b\omega_k-b\omega_j-\omega_kc_{nk}e^{it})}dt\\
 &=\frac{1}{4\pi i}\int_{\mathbb{T}}\frac{b^{n+\alpha}\omega_k^{-\alpha-1}((b+2+c_{nk}\omega)^{n-\alpha-1}-(b-c_{nk}\omega)^{n-\alpha-1})}{\omega\prod\limits_{j=1}^n((b+2)\omega_k-b\omega_j+\omega_kc_{nk}\omega)\prod\limits_{j\ne k}(b\omega_k-b\omega_j-\omega_kc_{nk}\omega)}d\omega\\
 &+\frac{1}{4\pi i}\int_{\mathbb{T}}\frac{b^{\alpha}\omega_k^{-\alpha-1}((b+2+ c_{nk}\omega)^{\alpha+1})((b+1)^2-(c_{nk}\omega+1)^2)}{\omega\prod\limits_{j=1}^n((b+2)\omega_k-b\omega_j+\omega_kc_{nk}\omega)\prod\limits_{j\ne k}(b\omega_k-b\omega_j-\omega_kc_{nk}\omega)}d\omega \\
 &-\frac{1}{4\pi i}\int_{\mathbb{T}}\frac{b^{\alpha}\omega_k^{-\alpha-1}(b-c_{nk}\omega)^{\alpha+1})((b+1)^2-(c_{nk}\omega+1)^2)}{\omega\prod\limits_{j=1}^n((b+2)\omega_k-b\omega_j+\omega_kc_{nk}\omega)\prod\limits_{j\ne k}(b\omega_k-b\omega_j-\omega_kc_{nk}\omega)}d\omega
		\end{aligned}
 \end{equation}

 Since $c_{nk}=\frac{1}{2}\min\{\min\limits_{j\ne k} b\lvert \omega_k-\omega_j\rvert,\min\limits_{j}\lvert (b+2)\omega_k-b\omega_j\rvert\},$ the only residue of the integral is the residue at $\omega=0$.

 Thus,
 \begin{align*}
     \eqref{need-more-cal}&=\frac{b^{n+\alpha}\omega_k^{n-\alpha-1}((b+2)^{n-\alpha-1}-b^{n-\alpha-1})}{2\prod\limits_{j=1}^n((b+2)\omega_k-b\omega_j)\prod\limits_{j\ne k}(b\omega_k-b\omega_j)}\\
     &+\frac{b^{\alpha}\omega_k^{2n-\alpha-1}(b^2+2b)^{n-\alpha-1}((b+2)^{\alpha+1}-b^{\alpha+1})}{2\prod\limits_{j=1}^n((b+2)\omega_k-b\omega_j)\prod\limits_{j\ne k}(b\omega_k-b\omega_j)}\\
     &=\frac{\omega^{-\alpha-1}}{2}\frac{b^{n-1}[b^{\alpha+1}(b+2)^{n-\alpha-1}-b^{n}+(b+2)^{n}-(b+2)^{n-\alpha-1}b^{\alpha+1}]}{n\omega_k^{n-1}b^{n-1}((b+2)^n-b^n)}\\
     &=\frac{\omega_k^{-\alpha}}{2n}=:a_{k\alpha}
 \end{align*}

    Suppose that $(\widetilde{\phi_{\alpha}}\otimes tr)(\omega_R(z))$, $0\leq \alpha \leq n-1$, are linearly dependent. Then, $(\phi_{\alpha}\otimes tr)(\omega_R(z))$ are linearly dependent in $H^1_{de}(P^c(R_{\lambda}),\mathbb{C})$. Thus, $\exists \;b_{\alpha}\in \mathbb{C}, 0\leq \alpha \leq n-1$, $s.t.$ 
    \begin{equation}\label{linear-equation}
			\sum\limits_{\alpha=0}^{n-1}b_\alpha(\phi_{\alpha}\otimes tr)(\omega_R(z))=[0].
		\end{equation}
    Integrating \eqref{linear-equation} on $\gamma_k$, we have
        \[\sum\limits_{\alpha=0}^{n-1}b_\alpha a_{k\alpha}=0.\]
    Choosing $k = 0, 1, \cdots , n-1$ in turn, we obtain a homogeneous linear system of equations with a coefficient matrix $(a_{k\alpha})_{k,\alpha=0}^{n-1}$. Since 
    \[\text{det}((a_{k\alpha})_{k,\alpha=0}^{n-1})=(\frac{1}{2n})^{n}\prod_{0\leq j < k\leq n-1}(\frac{1}{\omega_k}-\frac{1}{\omega_j}),\] the only solution to the homogeneous linear system of equations is the zero solution, i.e. $\forall 0\leq \alpha \leq n-1, b_{\alpha}=0$. Thus $(\widetilde{\phi_{\alpha}}\otimes tr)(R^{-1}(z))$ are $n$ different elements in the cohomology group.
\end{proof}
\noindent \textbf{Remark 1}\quad Let $\gamma=\{z(t):0\leq t\leq 2\pi\}$ be a closed path in the resolvent set. 
By Theorem \ref{joint-spec}, $G_{\theta}^k(z)\ne 0$ on $P(R_{\lambda})^c$. If we define the winding number $W(\gamma)$ of $\gamma$ around the projective joint spectrum as the winding number of $\prod\limits_{k=1}^nG_{\theta}^k(\gamma)$ around 0, then
\[W(\gamma)=\sum\limits_{k=1}^n\frac{1}{2\pi i}\int_{G_{\theta}^k(\gamma)}\frac{d\omega}{\omega}=\sum\limits_{k=1}^n\frac{1}{2\pi i}\int_0^{2\pi}\frac{d}{dt}\log(G_{\theta}^k(z(t)))dt.\]
Thus, the integral 
\begin{align*}
    \frac{1}{2\pi i}\int_{\gamma}(\phi_{0}\otimes tr)(\omega_R(z))&=\frac{1}{2\pi i}\int_{\gamma}d\bigg(\frac{1}{4n\pi}\int_0^{2\pi}\log(\prod\limits_{k=1}^nG_{\theta}^{k}(z))d\theta\bigg)\\
    &=\frac{1}{4n\pi}\int_0^{2\pi}\bigg(\sum\limits_{k=1}^n\frac{1}{2\pi i}\int_0^{2\pi}\frac{d}{dt}\log(G_{\theta}^k(z(t)))dt\bigg)d\theta\\
    &=\frac{W(\gamma)}{2n}.
\end{align*}
This provides the explanation for why $a_{k0}$ is not an integer.
\vspace{\baselineskip}

\noindent \textbf{Remark 2}\quad In the above theorem, we only provide $n$ linear independent elements in the cohomology group. However, it still remains a question whether the cohomology group is $n$ generated.

\section*{Acknowledgements} The authors thank the referee for helpful comments which make this paper more readable.

\end{document}